\documentclass[english,1p]{amsart}
\usepackage[latin9]{inputenc}
\usepackage{geometry}
\usepackage{amsthm}
\usepackage{amsmath}
\usepackage{amssymb}
\usepackage{esint}
\usepackage{tikz}
\usepackage{hyperref}

\makeatletter
\numberwithin{equation}{section}
\numberwithin{figure}{section}

\theoremstyle{plain}
\newtheorem{thm}{Theorem}
  \theoremstyle{plain}
  \numberwithin{thm}{section}
  \newtheorem{cor}[thm]{Corollary}
  \theoremstyle{plain}
   \newtheorem{definition}[thm]{Definition}
  
  \theoremstyle{plain}
   \newtheorem{proposition}[thm]{Proposition}
  \theoremstyle{remark}
  \newtheorem{rem}[thm]{Remark}

  \def\Ddots{\mathinner{\mkern1mu\raise\p@
\vbox{\kern7\p@\hbox{.}}\mkern2mu
\raise4\p@\hbox{.}\mkern2mu\raise7\p@\hbox{.}\mkern1mu}}
\makeatother

\usepackage{babel}

\newcommand{\norm}[1]{\left\| #1 \right\|}
\newcommand{\mklm}[1]{\left\{ #1 \right\}}
\newcommand{\eklm}[1]{\left\langle #1 \right\rangle}

\renewcommand{\d}{\,d}
\newcommand{\N}{{\mathbb N}}
\newcommand{\Z}{{\mathbb Z}}

\newcommand{\R}{{\mathbb R}}

\newcommand{\D}{{\mathcal D}}

\renewcommand{\H}{{\mathcal H}}
\newcommand{\I}{{\mathcal I}}
\newcommand{\J}{{\mathcal J }}
\newcommand{\Jbb}{{\mathbb J }}
\newcommand{\M}{{\mathcal M}}

\renewcommand{\O}{{\mathcal O}}
\newcommand{\W}{{\mathcal W}}

\newcommand{\1}{{\bf 1}}
\renewcommand{\epsilon}{\varepsilon}

\renewcommand{\rho}{\varrho}

\newcommand{\Cinft}{{\rm C^{\infty}}}
\newcommand{\CT}{{\rm C^{\infty}_c}}

\renewcommand{\L}{{\rm L}}
\newcommand{\Lcal}{{\mathcal L}}

\renewcommand{\S}{{\mathcal S}}

\newcommand{\SO}{\mathrm{SO}}

\newcommand{\g}{{\bf \mathfrak g}}

\renewcommand{\t}{{\bf \mathfrak t}}

\newcommand{\vol}{\text{vol}\,}
\newcommand{\dist}{\text{dist}\,}

\newcommand{\Crit}{\mathrm{Crit}}

\DeclareMathOperator{\supp}{supp\,}

\DeclareMathOperator{\gd}{\partial}

\DeclareMathOperator{\grad}{grad}

\newcommand{\e}[1]{\,{\mathrm e}^{#1}\,}

\newcommand{\bdm}{\begin{displaymath}}
\newcommand{\edm}{\end{displaymath}}
\newcommand{\bq}{\begin{equation}}
\newcommand{\eq}{\end{equation}}
\newcommand{\bqn}{\begin{equation*}}
\newcommand{\eqn}{\end{equation*}}

\begin{document}
\author{Pablo Ramacher}
\email{ramacher@mathematik.uni-marburg.de}
\address{Philipps-Universit\"at Marburg, FB  12 Mathematik und Informatik, Hans-Meerwein-Str., 35032 Marburg}
\title{Addendum to ''The equivariant spectral function of an invariant elliptic operator''}
%\begin{keyword}
%Equivariant local Weyl law \sep $\L^p$-bounds and concentration of eigenfunctions  \sep stationary phase, caustics, and desingularization
%\MSC[2010] 14E15 \sep 35P20 \sep 42B70 \sep  57S15 \sep 57S17  \sep 58J40 \sep 58J50 \sep 58K55
%\end{keyword}

\begin{abstract} Let $M$ be a compact boundaryless Riemannian manifold, carrying an  effective and isometric action of a torus $T$,  and $P_0$ an invariant elliptic classical pseudodifferential operator on $M$. In this note,  we strengthen  the asymptotics  for the equivariant (or reduced) spectral function  of $P_0$  derived in \cite{ramacher16}, which are already sharp in the eigenvalue aspect, to become almost sharp in the isotypic aspect.  In particular, this leads to hybrid equivariant $\L^p$-bounds  for eigenfunctions that are almost sharp  in the eigenvalue and isotypic aspect. %These results are used in \cite{ramacher-wakatsuki17} to derive subconvex bounds for Hecke-Maass forms on compact arithimetic quotients of semisimple Lie groups that are among the strongest up to now. 
\end{abstract}

\maketitle

\setcounter{tocdepth}{1}
\tableofcontents{}

\section{Introduction}

Let $M$ be a closed $n$-dimensional Riemannian manifold with an effective and isometric action of a compact Lie group $G$. In this paper, we strenghten the asymptotics derived in \cite{ramacher16} for  the equivariant (or reduced) spectral function of  an invariant elliptic operator on $M$, which are already sharp in the eigenvalue aspect, to become also almost sharp in the isotypic aspect in case that $G=T$ is a torus, that is, a compact connected Abelian Lie group. In particular, if $T$ acts on $M$ with orbits of the same dimension,  we obtain hybrid equivariant $\L^p$-bounds for eigenfunctions that are almost sharp up to a logarithmic factor.
% In case that  singular orbits are present, we are able to describe the caustic   behaviour of the reduced spectral function as one approaches  orbits of singular type, relying on our recent work  \cite{ramacher10} on singular equivariant asymptotics  obtained via desingularization techniques. 
 
To explain our results, consider  an elliptic classical pseudodifferential operator 
\bqn 
P_0:\Cinft(M) \, \longrightarrow \, \L^2(M)
\eqn
of degree $m$ on $M$ acting on the Hilbert space of square integrable functions   on $M$  with the space of smooth functions on $M$ as domain. We assume that $P_0$ is positive and symmetric, so that it has a unique self-adjoint extension $P$, which  has discrete spectrum. Let $\mklm{E_\lambda}$ be a spectral resolution of $P$, and denote by $e(x,y,\lambda)$   the \emph{spectral function} of $P$ which is given by the Schwartz kernel of $E_\lambda$.
Further, assume that  $M$ carries an effective and isometric action of a compact Lie group $G$ with Lie algebra $\g$ and orbits of dimension less or equal $n-1$. %The group $G$ might be disconnected or even finite, though the case of interest is when $G$ is continuous. 
Suppose that $P$ commutes with the {left-regular representation} $(\pi,\L^2(M))$ of $G$ so that each eigenspace of $P$ becomes a unitary $G$-module. If $\widehat G$ denotes the set of equivalence classes of irreducible unitary representations of $G$,  the Peter-Weyl theorem asserts that 
 \bq
\label{eq:PW} 
\L^2(M)=\bigoplus_{\gamma \in \widehat G} \L^2_\gamma(M),
\eq
a Hilbert sum decomposition, where $\L^2_\gamma(M):=\Pi_\gamma \L^2(M)$ denotes the $\gamma$-isotypic component, and $\Pi_\gamma$  the corresponding projection.  Let $e_\gamma(x,y,\lambda)$ be the spectral function of the operator $P_\gamma:=\Pi_\gamma \circ P\circ \Pi_\gamma$, which is also called the \emph{reduced spectral function} of $P$. Further, let $\Jbb:T^\ast M \to \g^\ast$ denote the momentum map of the Hamiltonian $G$-action on $T^\ast M$, induced by the action of $G$ on $M$,  and write $\Omega:=\Jbb^{-1}(\mklm{0})$.  In \cite[Theorem 4.3]{ramacher16},  the \emph{equivariant local Wey law}
 \bqn
% \label{eq:14.05.2015}
\left |e_\gamma(x,x,\lambda)- \lambda^{\frac{n-\kappa_x}{m}} \frac{d_\gamma [\pi_{\gamma|G_x}:\1]}{(2\pi)^{n-\kappa_x}}  \int_{\{ (x,\xi) \in \Omega, \, p(x,\xi)< 1\}} \frac{ \d \xi}{\vol \O_{(x,\xi)}} \right | \leq C_{x,\gamma} \, \lambda^{\frac{n-\kappa_x-1}{m}}, \quad x \in M, 
\eqn 
was shown  as $\lambda \to +\infty $, where $\kappa_x:=\dim \O_x$ is the dimension of the $G$-orbit through $x$,  $d_\gamma$ denotes the dimension of an irreducible $G$-representation $\pi_\gamma$ belonging to $\gamma$ and $ [\pi_{\gamma|G_x}:\1]$ the multiplicity of the trivial representation  in the restriction of $\pi_\gamma$ to the isotropy group $G_x$ of $x$, while $C_{x,\gamma}>0$ is a constant satisfying
 \bq
 \label{eq:25.5.2018a}
 C_{x,\gamma} =O_x\big (  d_\gamma \sup_{l \leq \lfloor \kappa_x/2+3 \rfloor} \norm{\D^l \gamma}_\infty\big ), 
 \eq
 and $D^l$ are differential operators on $G$ of order $l$.  Both the leading term and the constant $C_{x,\gamma}$   in general depend in a highly non-uniform way on $x\in M$, exhibiting a caustic behaviour in the neighborhood of singular orbits. 
A precise description of this  caustic behaviour was achieved in \cite{ramacher16} by relying on the results \cite{ramacher10} on singular equivariant asymptotics  obtained via resolution of singularities.  More precisely, consider the stratification  $M=M(H_1) \, \dot \cup \dots \dot \cup \, M(H_L)$ of $M$ into orbit types, arranged in such a way that 
$(H_i) \leq (H_j)$ implies $i \geq j$, and let $\Lambda$ be the maximal length that a maximal totally ordered subset of isotropy types can have. Write  $M_\mathrm{prin}:=M(H_L)$, $M_\mathrm{except}$, and $M_\mathrm{sing}$ for the union of all orbits of principal, exceptional, and singular type, respectively, so that 
\bqn
%\label{eq:15.08.2016}
M= M_\mathrm{prin}\, \dot \cup \, M_\mathrm{except}\, \dot \cup \, M_\mathrm{sing},
\eqn
  and denote by $\kappa:=\dim G/H_L$ the dimension of an orbit of principal type.   Then, by   \cite[Theorem 7.7]{ramacher16} one has  for $x \in M_\mathrm{prin}\cup M_\mathrm{except}$  and $\lambda \to +\infty $ the \emph{singular equivariant local Weyl law}
   \begin{align*}
% \label{eq:17.06.2016}
 \begin{split}
 \Big |e_\gamma(x,x,\lambda)&- \frac{d_\gamma \lambda^{\frac{n-\kappa}{m}}}{(2\pi)^{n-\kappa}} \sum_{N=1}^{\Lambda-1} \,  \sum_{{i_1<\dots< i_{N} }} \, \prod_{l=1}^{N}   |\tau_{i_l}|^{\dim G- \dim H_{i_l}-\kappa}  \mathcal{L}^{0,0}_{i_1\dots i_{N} }(x,\gamma) \Big  |\\ 
&\leq \widetilde C_\gamma \lambda^{\frac{n-\kappa-1}m} \sum_{N=1}^{\Lambda-1}\, \sum_{{i_1<\dots< i_{N}}}   \prod_{l=1}^N     |\tau_{i_l}|^{\dim G- \dim H_{i_l}-\kappa-1}, 
\end{split}
\end{align*}
 where the multiple sums run over all possible maximal totally ordered subsets $\mklm{(H_{i_1}),\dots, (H_{i_N})}$ of singular isotropy types, the  coefficients $\mathcal{L}^{0,0}_{i_1\dots i_{N}}$  are explicitly given and bounded functions  in $x$, and  $\tau_{i_j} =\tau_{i_j}(x)\in (-1,1)$ are desingularization parameters that arise in the resolution process satisfying  $|\tau_{i_j}|\approx \dist (x, M(H_{i_j}))$, while  $\widetilde C_\gamma>0$ is a constant independent of $x$ and $\lambda$ that fulfills
 \bq
 \label{eq:25.5.2018b}
\widetilde  C_{\gamma} =O\big (  d_\gamma \sup_{l \leq \lfloor \kappa/2+3 \rfloor} \norm{\D^l \gamma}_\infty\big ).
 \eq 
  
As a major consequence, the above expansions lead to  equivariant bounds for  eigenfunctions. In the non-singular case, that is, when only principal and exceptional orbits are present,  and consequently all $G$-orbits have the same dimension $\kappa$,   the hybrid $\L^q$-estimates 
\bq
\label{eq:Lqbound}
\norm{u}_{\L^q(M)} \leq \begin{cases} C_\gamma \,  \lambda^{\frac{\delta_{n-\kappa}(q)}{m}} \norm{u}_{\L^2}, &  \frac{2(n-\kappa+1)}{n-\kappa-1} \leq q \leq \infty, \vspace{2mm} \\ C_\gamma \, \lambda^{\frac{(n-\kappa-1)(2-q')}{4m q'}} \norm{u}_{\L^2}, &  2 \leq q \leq \frac{2(n-\kappa+1)}{n-\kappa-1}, \end{cases} 
\eq
were shown in \cite[Theorem 5.4]{ramacher16} for any eigenfunction   $u \in \L^2_\gamma(M)$ of $P$   with eigenvalue $\lambda$,  where $\frac 1q+\frac 1{q'}=1$, $\delta_n(p):=\max \left ( n \left |1/ 2 - 1/p \right | -1/2,0 \right )$,  and $C_{\gamma}>0$ is a constant independent of $\lambda$ satisfying the estimate
 \bq
 \label{eq:24.07.2017}
 C_\gamma \ll \sqrt{d_\gamma \sup_{l \leq \lfloor \kappa/2+1\rfloor} \norm{D^l\gamma}_\infty},
 \eq
provided that  the co-spheres $S_x^\ast M$ are strictly convex. Note that for the proof of $\L^p$-bounds it is necessary to describe the caustic behaviour of the relevant spectral kernels  as $\mu\to +\infty $ in  a neighborhood of the diagonal, which makes things considerably more envolved. 
%Using complex interpolation techniques, we then prove in Theorem \ref{thm:20.02.2016} the hybrid bounds in the eigenvalue and isotypic aspect
%\bqn
%\norm{(\chi_\lambda \circ \Pi_\gamma) u}_{\L^q(M)} \leq \begin{cases} C_{\gamma} \,  \lambda^{\frac{\delta_{n-\kappa}(q)}{m}} \norm{u}_{\L^2(M)}, &  \frac{2(n-\kappa+1)}{n-\kappa-1} \leq q \leq \infty, \vspace{2mm} \\ C_{\gamma} \, \lambda^{\frac{(n-\kappa-1)(2-q')}{4m q'}} \norm{u}_{\L^2(M)}, &  2 \leq q \leq \frac{2(n-\kappa+1)}{n-\kappa-1}, \end{cases} 
%\eqn
% where $\frac 1q+\frac 1{q'}=1$, 
% and $C_\gamma$ is as in \eqref{eq:24.07.2017}.   In particular, we have the hybrid equivariant bound
%\bqn 
%\norm{u}_{\L^q(M)} \leq \begin{cases} C_\gamma \,  \lambda^{\frac{\delta_{n-\kappa}(q)}{m}}, &  \frac{2(n-\kappa+1)}{n-\kappa-1} \leq q \leq \infty, \vspace{2mm} \\ C_\gamma \, \lambda^{\frac{(n-\kappa-1)(2-q')}{4m q'}}, &  2 \leq q \leq \frac{2(n-\kappa+1)}{n-\kappa-1}, \end{cases} 
%\eqn
%for any eigenfunction of $P$  belonging to  $u \in \E_\lambda \cap \L^2_\gamma(M)$ and satisfying $\norm{u}_{\L^2}=1$, 
%provided that $G$ acts on $M$ with  orbits of the same dimension $\kappa$. 
In case that singular orbits are present,   one has the pointwise bound
\bq
\label{eq:4.12.2015}
\sum_{\stackrel{\lambda_j \in (\lambda,\lambda+1],}{ e_j \in \L^2_\gamma(M)}} |e_j(x)|^2  \leq \begin{cases}  C \, \lambda^{\frac{n-1}m}, & \hspace{-.0cm} x\in M_\mathrm{sing}, \\
& \\
\widetilde C_\gamma \, \lambda^{\frac{n-\kappa-1}m} \sum\limits_{N=1}^{\Lambda-1}\, \sum\limits_{{i_1<\dots< i_{N}}}   \prod\limits_{l=1}^N  |\tau_{i_l}|^{\dim G- \dim H_{i_l}-\kappa-1}, & x\in M- M_\mathrm{sing},  \end{cases}
\eq
for a constant  $C>0$  independent of $\gamma$,   where  $\mklm{e_j}_{j \geq 0}$ is an orthonormal basis of $\L^2(M)$  compatible with the decomposition \eqref{eq:PW}, 
%In comparison with the bound \eqref{eq:clust}, where the dependency of the constant $C_{x,\gamma}$ on $x$ remains unspecified, the bound \eqref{eq:4.12.2015} gives a rather precise description of the growth of eigenfunctions near singular orbits. 
 showing that eigenfunctions tend to concentrate along lower dimensional orbits.

The aim of this note is to sharpen the above results in the isotypic aspect in case that $G=T$ is a torus, and show that instead of the bounds \eqref{eq:25.5.2018a} and  \eqref{eq:25.5.2018b} one has the better estimates
\bqn
C_{x,\gamma}=O_x\Big (\sup_{l \leq 1 }\norm{D^l \gamma}_\infty\Big ), \qquad \widetilde C_{\gamma}=O\Big (\sup_{l \leq 1 }\norm{D^l \gamma}_\infty\Big ), \qquad \gamma \in \W_\lambda,
\eqn
where $\W_\lambda$ denotes the set of representations
\bqn
\W_\lambda:=\mklm{\gamma \in \widehat T' \mid |\gamma | \leq \frac{\lambda^{1/m}}{\log \lambda}}.
\eqn
Here $\widehat T'\subset$ stands for the subset of representations occuring in the Peter-Weyl decomposition \eqref{eq:PW}, and we denoted the differential  of a character $\gamma\in \widehat T$, which corresponds to an integral linear form $\gamma:\t \rightarrow i\R$,  by the same letter. 
Similarly, it will be shown that the constant $C_\gamma$ in \eqref{eq:24.07.2017} actually satisfies the bound
 \bqn
 C_\gamma \ll 1, \qquad \gamma \in \W_\lambda. 
 \eqn
 By the equivariant Weyl law \cite{ramacher10} and Gauss' law,  $|\gamma|$ can grow at most of rate $\lambda^{1/m}$. 
 %In other words, the $L^q$-norm of an eigenfunction can grow at most as $\sqrt{d_\gamma}$. 
Thus, the bounds \eqref{eq:Lqbound} hold for \emph{almost any} eigenfunction $u\in \L^2(M)$ with $C_\gamma$ independent of $\gamma$, which is consistent with recent results of Tacy \cite{tacy18}.
%\footnote{\bf Note that Tacy's result do not includ`e estimates off the diagonal.  Mention potential generalization of the bound of Bernstein-Reznikov, Sobolev norms of automorphic functionals.} 
%, improving the bounds of Seeger and Sogge \cite{seeger-sogge} considerably. 
As will be discussed, the improved bounds are almost sharp in this sense, being already attained for $\SO(2)$-actions on the $2$-sphere and the $2$-torus. For their proof, a careful examination of the remainder in the stationary phase expansion of the relevant spectral kernels is necessary.  These bounds are crucial for deriving  hybrid subconvex bounds  for Hecke-Maass forms on compact arithimetic quotients of semisimple Lie groups in the eigenvalue and isotypic aspect \cite{ramacher-wakatsuki17}. % by relying on the asymptotic description  of the integrals \eqref{eq:1.12.2015} given in Theorems \ref{thm:12.05.2015} and \ref{thm:14.05.2017}.  \\

Through the whole document, the notation $O(\mu^{k}), k \in \R \cup \mklm{\pm \infty},$ will mean an  upper bound of the form $C \mu^k$  with a constant $C>0$ that is uniform in all relevant variables, while $O_\aleph(\mu^{k})$ will denote an upper bound of the form $C_\aleph \,  \mu^k$ with a constant $C_\aleph> 0$ that depends on the indicated variable $\aleph$. In the same way,  we shall write $a\ll_\aleph b$ for two real numbers $a$ and $b$, if there exists a constant $C_\aleph>0$ depending only on $\aleph$ such that $|a| \leq C_\aleph b$, and similarly $a \ll b$, if the bound is uniform in all relevant variables. Finally, $\N$ will denote the set of natural numbers $0,1,2,3,\dots$. \\

%{\bf Acknowledgements.}

\section{The reduced spectral function of an invariant elliptic operator}
\label{sec:RSF}

Let $M$ be  a closed connected Riemannian manifold  of dimension $n$ with Riemannian volume density $dM$, and $P_0$ an elliptic classical pseudodifferential operator on $M$ 
of degree $m$ which  is positive and symmetric.  The  principal symbol $p(x,\xi)$ of $P_0$  constitutes a strictly positive function on $T^\ast M\setminus\mklm{0}$, where  $T^\ast M$ denotes the cotangent bundle  of $M$. The operator $P_0$ has a unique self-adjoint extension $P$, its domain being the $m$-th Sobolev space $H^m(M)$. It is well known that there exists  an orthonormal basis  $\mklm{e_j}_{j\geq 0}$ of $\L^2(M)$ consisting of eigenfunctions of $P$ with eigenvalues $\mklm{\lambda_j}_{j \geq 0}$ repeated according to their multiplicity, and that $Q:=\sqrt[m]{P}$ constitutes  a classical pseudodifferential operator of order $1$ with principal symbol $q(x,\xi):=\sqrt[m]{p(x,\xi)}$ and domain  $H^1(M)$. Again, $Q$ has discrete spectrum, and its eigenvalues  are given by $\mu_j:=\sqrt[m]{\lambda_j}$.  The spectral function $e(x,y,\lambda)$ of $P$ can then be described by  studying the spectral function of $Q$, which in terms of the basis $\mklm{e_j}$ is given by
\bqn 
e(x,y,\mu):=\sum_{\mu_j\leq \mu} e_j(x) \overline{e_j(y)}, \qquad \mu\in \R, 
\eqn
and belongs to $\Cinft(M \times M)$ as a function of $x$ and $y$. Let $\chi_\mu$ be the spectral projection onto the sum of eigenspaces of $Q$ with eigenvalues in the interval  $(\mu, \mu+1]$, and denote its Schwartz kernel by $\chi_\mu(x,y):=e(x,y,\mu+1) - e(x,y,\mu)$. To obtain an asymptotic description of the spectral function of $Q$ let $\rho \in \S(\R,\R_+)$ be such that $\rho(0)=1$ and $\supp \hat \rho\in (-\delta/2,\delta/2)$ for a given $\delta>0$, and define the {approximate spectral projection operator} 
\bq
\label{eq:2.1} 
\widetilde \chi_\mu u := \sum_{j=0}^\infty \rho(\mu-\mu_j) E_{j} u, \qquad u \in \L^2(M),
\eq
where $E_j$ denotes the orthogonal projection onto the subspace spanned by $e_j$. Clearly, $K_{\widetilde \chi_\mu}(x,y):=\sum_{j=0}^\infty \rho(\mu-\mu_j) e_j(x) \overline{e_j(y)}\in \Cinft(M\times M)$ constitutes the kernel of $\widetilde \chi_\mu$. 
%Now, notice that for $\mu,\tau\in \R$ one has
%\bqn 
%\rho(\mu -\tau) = \frac 1 {2\pi} \intop_\R \hat \rho(t) e^{-it\tau} e^{it\mu} \d t, 
%\eqn
%where $\hat \rho(t)$ denotes the Fourier transform of $\rho$, 
%so that for $u \in \L^2(M)$ we obtain
%\begin{align*}
%\widetilde \chi_\mu u=\frac 1 {2\pi} \sum_{j=0}^\infty \intop_\R \hat \rho(t)  e^{it\mu}  e^{-it\mu_j}\d t \, E_j u= \frac 1 {2\pi}  \intop_\R \hat \rho(t)  e^{it\mu}  U(t) u \d t, 
%\end{align*}
%where $U(t)$ denotes the one-parameter group of unitary operators in $\L^2(M)$ 
%  \bqn 
% U(t)=\int e^{-it\mu} dE_\mu^Q=e^{-itQ}, \qquad t \in \R,
% \eqn
% given by the Fourier transform of the spectral measure, 
% $\{E_\mu^Q\}$ being a spectral resolution of $Q$. The central result of H\"ormander \cite{hoermander68}  then says that  $U(t)=e^{-itQ}:\L^2(M)\rightarrow \L^2(M)$ can be approximated by Fourier integral operators, 
As H\"ormander \cite{hoermander68} showed,  $\widetilde \chi_\mu$ can be approximated by Fourier integral operators yielding an asymptotic formula for the kernels of $\widetilde \chi_\mu$ and  $\chi_\mu$, and finally for the spectral function of $Q$ and $P$. 
 
Now, assume that  $M$ carries an effective and isometric action of a compact   Lie group $G$. Let $P$ commute with the left-regular representation $(\pi,\L^2(M))$ of $G$. Consider  the  Peter-Weyl decomposition  of $\L^2(M)$, and let $\Pi_\gamma$  be the projection onto the isotypic component belonging to $\gamma \in \widehat G$, which is given by the Bochner integral
\bqn 
\Pi_\gamma=d_\gamma \intop_G \overline{\gamma(g)} \pi(g) \d_G(g),
\eqn 
where $d_\gamma$ is the dimension of an unitary irreducible  representation of class $\gamma$, and  $d_G(g) \equiv dg$  Haar measure on $G$, which we assume to be normalized such that $\vol G=1$. If $G$ is finite, $d_G$ is simply the counting measure.  In addition, let us suppose that the orthonormal basis $\mklm{e_j}_{j\geq 0}$ is  compatible with the Peter-Weyl decomposition in the sense that each vector $e_j$ is contained in some isotypic component $\L^2_\gamma(M)$. In order to describe   the spectral function of the operator $Q_\gamma:=\Pi_\gamma \circ Q\circ \Pi_\gamma=Q\circ \Pi_\gamma=\Pi_\gamma \circ Q$ given by
\bq
\label{eq:24.09.2015}
e_\gamma (x,y,\mu):=\sum_{\mu_j\leq \mu,\, e_j \in \L^2_\gamma(M)} e_j(x) \overline{e_j(y)},
\eq
we consider   the composition $ \chi_\mu\circ \Pi_\gamma$ with kernel 
$
K_{\chi_\mu \circ \Pi_\gamma}(x,y)=e_\gamma(x,y,\lambda+1)-e_\gamma(x,y,\lambda)
$, 
together with  the corresponding equivariant approximate spectral projection
\begin{align}
\label{eq:1004}
(\widetilde \chi_\mu \circ \Pi_\gamma) u = \sum_{j\geq 0,\, e_j \in \L^2_\gamma(M)} \rho(\mu-\mu_j) E_{j} u.%= \frac {d_\gamma} {2\pi} \intop_G  \intop_\R \hat \rho(t)  e^{it\mu}  \overline{\gamma(g)} \big ( U(t)\circ \pi(g) \big ) u \d t \d g.
\end{align}
Its kernel can be written as 
\bqn 
K_{\widetilde \chi_\mu \circ \Pi_\gamma}(x,y):=\sum_{j\geq 0, e_j \in \L^2_\gamma(M)} \rho(\mu-\mu_j) e_j(x) \overline{e_j(y)}\in \Cinft(M\times M).
\eqn
By using Fourier integral operator methods, it was shown in \cite{ramacher16} that  the kernel of $\widetilde \chi_\mu \circ \Pi_\gamma$ can be expressed as follows. Let $\mklm{(\kappa_\iota, Y_\iota)}_{\iota \in I}$, $\kappa_\iota:Y_\iota \stackrel{\simeq}\to \widetilde Y_\iota \subset \R^n$, be an atlas for $M$,  $\mklm{f_\iota}$ a corresponding partition of unity, and $\mklm{\bar f_\iota}$ a set of test functions with compact support in $Y_\iota$ satisfying $\bar f_\iota \equiv 1$ on $\supp f_\iota$. 
Consider further a test function  $0 \leq \alpha \in \CT(1/2, 3/2)$  such that $\alpha \equiv 1$ in a neighborhood of $1$, and set
\begin{align}
\begin{split}
\label{eq:02.05.2015}
 I^\gamma_\iota(\mu, R, s, x,y):= &   \int _G  \int_{\Sigma^{R,s}_{\iota,x}} e^{i{ \mu}  \Phi_{\iota,x,y}(\omega,g)} \hat \rho(s)  \overline{\gamma(g)} f_\iota( x) \\ &\cdot     a_\iota(s, \kappa_\iota(  x) , \mu \omega)  \bar f _\iota (g \cdot y)   \alpha(q( x, \omega)) J_\iota(g,y) {\d\Sigma^{R,s}_{\iota,x}(\omega) \d g},
 \end{split}
\end{align}
where $\Phi_{\iota,x,y}(\omega,g):=\eklm{\kappa_\iota(  x) - \kappa_\iota(g \cdot  y),\omega}$,  $a_\iota \in S^0_{\mathrm{phg}}$ is a suitable classical polyhomogeneous symbol satisfying $a_\iota(0,\tilde x, \eta)=1$, $J_\iota(g,y)$ a Jacobian, and 
\bq
\label{eq:20.04.2015}
\Sigma^{R,s}_{\iota,x}:=\mklm{\omega \in \R^n \mid  \zeta_\iota (s, \kappa_\iota (x),\omega) = R}
\eq
is a smooth compact hypersurface given in terms of  a smooth function  $\zeta_\iota$  which is  homogeneous in $\eta$ of degree $1$ and satisfies $\zeta_\iota(0, \tilde x, \eta) = q(\kappa_\iota^{-1}(\tilde x), \eta)$. %
Then, by \cite[Corollary 2.2]{ramacher16} one has for $\mu \geq 1$,  $x,y \in M$, and each $\tilde N \in \N$  the asymptotic expansion 
\begin{align}
\label{eq:13.06.2016}
K_{\widetilde \chi_\mu \circ \Pi_\gamma}(x,y)
= &\Big(\frac{\mu}{2\pi}\Big )^{n-1}  \frac{d_\gamma}{2\pi} \sum _\iota \Big [   \sum_{j=0}^{\tilde N-1} D^{2j}_{R,s} I^\gamma_\iota(\mu, R, s,x,y)_{|(R,s)=(1,0)} \, \mu^{-j} +  \mathcal{R}^\gamma_{\iota}(\mu,x,y) \Big ]
\end{align}
up to terms
%\footnote{\bf Eventually expand...} 
of order $O(|\mu|^{-\infty}\norm{\gamma}_\infty)$ which are uniform in $x,y$, where $D^{2j}_{R,s}$ are known differential operators of order $2j$ in $R,s$, and 
\begin{align*}
|\mathcal{R}^\gamma_{\iota}(\mu,x,y)| \leq& C\mu^{-\tilde N} \sum_{|\beta| \leq 2\tilde N +3} \sup_{R,s} \big |\gd_{R,s}^\beta  I^\gamma_\iota(\mu,R,s,x,y)  \big |
\end{align*}
for some constant $C>0$. On the other hand,  $K_{\widetilde \chi_\mu \circ \Pi_\gamma}(x,y)$ is rapidly decaying as  $\mu \to -\infty$ and uniformly bounded in $x,y$ by $\norm{\gamma}_\infty$.

\section{Equivariant asymptotics of oscillatory integrals}

Let the notation be as in the previous section. As we have seen there, the question of describing the spectral function  in the equivariant setting reduces to the study of   oscillatory integrals of the form 
\bq
\label{eq:03.05.2015}
I^\gamma_{x,y}(\mu):=\int_{G}\int_{\Sigma^{R,s}_x} e^{i\mu \Phi_{x,y}(\omega,g)} \overline{\gamma(g)} a (x,y, \omega,g) \d \Sigma^{R,s}_x (\omega) \d g, \qquad  \mu \to + \infty,
\eq
with $\Sigma^{R,s}_x$ as in \eqref{eq:20.04.2015} and phase function 
\bqn
\Phi_{x,y}(\omega,g):=  \eklm{\kappa(x) - \kappa( g\cdot y), \omega},
\eqn 
where we have skipped the index $\iota$ for simplicity of notation, and  $a \in \CT$ is an amplitude that might depend on $\mu$  and other parameters such that $(x,y,\omega, g) \in \supp a$ implies $x, g \cdot y \in Y$. In what follows, we shall  write $^yG:=\mklm{g \in G \mid g\cdot y \in Y}$, as well as  
 \bq
 \label{eq:11.9.2017} 
 I^\gamma_x(\mu) :=I^\gamma_{x,x}(\mu), \qquad \Phi_x:=\Phi_{x,x}.
 \eq
%The asymptotic behaviour of these integrals is related to that  of  oscillatory integrals of the form
%\bq
%\label{eq:integral} 
%I(\mu)= \intop_G\intop_{T^\ast Y} e^{i\mu \Phi(x,\eta,g)} a (x,\eta,g) \d (T^\ast Y)(x,\eta) \d g, \qquad \mu \to + \infty,
%\eq
%with  phase function
%\bq
%\label{eq:phase}
%\Phi(x,\eta,g):= \eklm{\kappa(  x) - \kappa(g \cdot  x),\eta},
%\eq
%and suitable amplitude $a$. 
Let us assume in the following that $G$ is a continuous group, and write  $\kappa(x)=(\tilde x_1, \dots, \tilde x_n)$ so that the canonical local trivialization of $T^\ast Y$ reads
\bqn 
Y \times \R^n \, \ni (x,\eta) \quad \equiv \quad \sum_{k=1}^n \eta_k (d\tilde x_k)_{x} \in \, T^\ast_xY.
\eqn
With respect to this trivialization, we shall identify  $\Sigma^{R,s}_{x'}$ with a subset in $T^\ast_{x} Y$ for eventually different $x$ and $x'$, if convenient. 
Let $\Omega:=\Jbb^{-1}(\mklm{0})$ be the zero level set of the momentum map $\Jbb: T^\ast M \to \g^\ast$ of the underlying Hamiltonian $G$-action on $T^\ast M$. 
%Since
% \bq
%\label{eq:Ann}
%(x,\eta) \in \Omega  \cap T^\ast _xM \quad \Longleftrightarrow \quad (x,\eta)  \in \mathrm{Ann}(T_x (G\cdot x)),
%\eq
%where $\mathrm{Ann} \, (V_x) \subset T_x^\ast M$ denotes the annihilator of a vector subspace $V_x \subset T_xM$,  a simple computation shows that the critical set of $\Phi $ is  given by  
%\bq
%\label{eq:23.04.2015}
%\Crit \, \Phi=\mklm{(x,\eta,g) \in T^\ast Y \times G\mid d(\Phi)_{(x,\eta,g)}=0} =\mklm{(x,\eta,g) \in ( \Omega \cap T^\ast Y) \times G\mid g \in G_{(x,\eta)}}.
%\eq
%In what follows, we shall compute  the critical set of the phase function $\Phi_{x,y}$, which is much more involved. 
 Let $\O_x:=G\cdot x$ denote the $G$-orbit and $G_x:=\mklm{g \in G\mid g\cdot x=x}$ the stabilizer or isotropy group of a point  $x\in M$. Throughout the paper, it is assumed that 
\bqn 
 \dim \O_x \leq n-1 \qquad \text{for all } x \in M.
\eqn
Let further $N_y\O_x$ be the normal space to the orbit $\O_x$ at a point $y \in \O_x$, which can be identified with  $\mathrm{Ann}(T_y\O_x)$ via the underlying Riemannian metric. For $x\in Y$ and $\O_y \cap Y \not=\emptyset$ let
\bqn 
\Crit \, \Phi_{x,y}:=\Big \{(\omega,g) \in \Sigma^{R,s}_x  \times \, ^y G\mid \, d(\Phi_{x,y})_{(\omega,g)}=0\Big \}
\eqn
be the critical set of $\Phi_{x,y}$. With $M_\text{prin}$, $M_\text{except}$, and $M_\text{sing}$ denoting the principal, exceptional, and singular stratum, respectively, it was shown in \cite[Lemma 3.1]{ramacher16} that  
%\label{lem:21.04.2015}
\begin{itemize}
\item if  $y\in \mathcal{O}_x$, the set $\Crit \, \Phi_{x,y}$ is clean and given by the smooth submanifold
\bqn
%\label{eq:22.09.2015}
\J=\big \{(\omega,g)\mid (g \cdot y, \omega) \in \Omega,  \,  x=g\cdot y\big \}= V_\J \times G_\J
\eqn
of  codimension $2\dim \O_x$, with $V_\J=\Sigma^{R,s}_x \cap N_x\O_x$ and $G_\J=\mklm{g \in G \mid x=g\cdot y}\subset \, ^y G$.

\item if $y \not\in  \mathcal{O}_x$, 
$$\Crit \, \Phi_{x,y}=\Big \{(\omega,g)\mid (g \cdot y,\omega) \in \Omega, \,  \kappa(x)-\kappa(g \cdot y) \in N_\omega \Sigma^{R,s}_x\Big \};$$
furthermore, assume that $G$ acts on $M$ with orbits of the same dimension $\kappa$, that is, $M=M_\mathrm{prin}\,  \cup\,  M_\mathrm{except}$, and that the co-spheres $S_x^\ast M$ are strictly convex. Then, either $\Crit \, \Phi_{x,y}$ is empty, or,  choosing $Y$ sufficiently small,  $\Crit \, \Phi_{x,y}$ is clean and of codimension $n-1+\kappa$, its finitely many connected components being of the form 
\bqn 
\J=V_\J \times G_\J
\eqn
with  $V_\J= \mklm{\omega_\J}$ and $G_\J = g_\J \cdot G_y\subset \, ^y G$ for some $\omega_\J\in \Sigma^{R,s}_x$ and $g_\J \in G$.  
\end{itemize}
 From this an asymptotic expansion for the integrals $I^\gamma_{x,y}(\mu)$ was deduced in \cite[Theorem 3.3]{ramacher16}, yielding a corresponding asymptotic formula for $K_{\widetilde \chi_\mu \circ \Pi_\gamma}(x,y)$. In this paper, we improve the estimate for the remainder in the  isotypic aspect in case that $G=T$ is a torus, which we assume from now on. 
 
  For this, recall that the exponential function $\exp$ is a covering homomorphism of $\t$ onto $T$, and its kernel $L$ a lattice in $\t$. Let $\widehat T$ denote the  \emph{set of characters of $T$}, that is, of all continuous homomorphisms of $T$ into the circle, which we identify with the unitary dual of $T$.  The differential of a character $\gamma: T \to S^1$, denoted by the same letter, is a linear form  $\gamma:\t\to i \R$ which is \emph{integral} in the sense that $\gamma(L) \subset 2\pi i \, \Z$. On the other hand, if $\gamma$ is an integral linear form, one defines
\bqn 
t^\gamma= e^{\gamma(X)}, \qquad t= \exp X \in T,
\eqn
setting up an identification of $\widehat T$ with the integral linear forms on $\t$ via $\gamma(t)\equiv t^\gamma$. Further, all irreducible representations of $T$ are $1$-dimensional. We now make the following 

\begin{definition}
Denote by $\widehat T'\subset \widehat T$ the subset of representations occuring in the decomposition \eqref{eq:PW} of $\L^2(M)$, and let $\mklm{\mathcal{V}_\mu}_{\mu \in (0,\infty)}$ be a family of finite subsets $\mathcal{V}_\mu\subset \widehat T'$ such that 
\bqn 
\max_{\gamma \in \mathcal{V}_\mu} |\gamma| \leq C \frac\mu{\log \mu}
\eqn
for a constant $C>0$ independent of $\mu$. 
\end{definition}

Our main result is the following improvement of the remainder and coefficient estimates in \cite[Theorem 3.3]{ramacher16}.

\begin{thm}
\label{thm:12.05.2015} 
Assume that $T$ is a torus acting on $M$ with orbits of dimension less or equal $n-1$, and let $ \mathcal{V}_\mu$ be as in the previous definition. 

 \begin{enumerate}
 \item[(a)] Let $y \in \mathcal{O}_x$. Then, for every $\gamma \in \widehat T$ and $\tilde N=0,1,2,\dots $ one has the asymptotic formula
\bqn
  I^\gamma_{x,y}(\mu)=(2\pi/\mu)^{\dim \mathcal{O}_x} \left [\sum_{k=0}^{\tilde N-1} \mathcal{Q}_{k}(x,y) \mu^{-k}  +\mathcal{R}_{\tilde N}(x,y,\mu)\right ],  \qquad \mu \to +\infty,
\eqn
where the  coefficients and the  remainder depend smoothly on $R$ and $s$. The coefficients satisfy the bounds
\begin{align*}
|\mathcal{Q}_k(x,y)|&\leq  C_{k,\Phi_{x,y}} \vol (\supp a(x,y,\cdot,\cdot)\cap \mathcal{C}_{x,y}) \sup _{l\leq k} \norm{(D_\omega^{2l} D_t^l \gamma a)(x,y, \cdot,\cdot)}_{\infty}
\end{align*}
while the remainder satisfies
\begin{align*}
|\mathcal{R}_{\tilde N}(x,y,   \mu) | &\leq \widetilde C_{\tilde N,\Phi_{x,y}}  \vol (\supp a  (x,y,\cdot,\cdot ))  \\ & \cdot \sup_{l\leq  2\tilde N+ \dim \mathcal{O}_x +1} \norm{(D_\omega^l D_t^l  a)(x,y,\cdot,\cdot )}_{\infty} \,  \sup_{l\leq  \tilde N} \norm{ D_t^l  \gamma}_{\infty}  \mu^{-\tilde N}, \qquad \gamma \in \mathcal{V}_\mu.
\end{align*}
The bounds are uniform in $R,s$ for suitable constants $C_{k,\Phi_{x,y}}>0$ and  $\widetilde C_{\tilde N,\Phi_{x,y}}>0$, where $D_\omega^l$ and $D_t^l$ denote  differential operators of order $l$ on $\Sigma^{R,s}_x$ and $T$, respectively. As functions in $x$ and $y$, $\mathcal{Q}_k(x,y)$ and $\mathcal{R}_{\tilde N}(x,y,\mu)$ are smooth on  $Y \cap M_\mathrm{prin}$, and the constants $C_{k,\Phi_{x,y}}$ and  $\widetilde C_{\tilde N,\Phi_{x,y}}$ are uniformly bounded in $x$ and $y$ if  $M= M_\mathrm{prin} \cup M_\mathrm{except}$.
% In particular, 
%\bqn
%%\label{eq:L0}
%\mathcal{Q}_0(x,y)=\int_{ \mathcal{C}_{x,y}} \frac { \overline{\gamma(g)} a(  x , y, \omega,g) }{|\det   \, \Phi_{x,y}''(\omega,g)_{N_{( \omega, g)} \mathcal{C}_{x,y}}|^{1/2}} \d\mathcal{C}_{x,y}(\omega,g),
%\eqn
%where   $d \mathcal{C}_{x,y}$ denotes the induced volume density.

\item[(b)] Let $y \not \in \mathcal{O}_x$. Assume that  $M=M_\mathrm{prin}\,  \cup\,  M_\mathrm{except}$ and that the co-spheres $S_x^\ast M$ are strictly convex.  Then, for sufficiently small $Y$ and every $\tilde N=0,1,2,\dots $   one has the asymptotic formula
\bqn
  I^\gamma_{x,y}(\mu)=\sum_{\J \in \pi_0(\Crit \, \Phi_{x,y})}(2\pi/\mu)^{\frac{n-1+\kappa}{2}} e^{i\mu \,^0\Phi_{x,y}^\J} \left [  \sum_{k=0}^{\tilde N-1} \mathcal{Q}_{\J,k}(x,y) \mu^{-k}  +\mathcal{R}_{\J, \tilde N}(x,y,\mu)\right ]
\eqn
as $\mu \to +\infty$, where $\kappa:=\dim M/T$ and  $^0\Phi_{x,y}^\J$ stands for  the constant values of $\Phi_{x,y}$ on the connected components $\J$ of  its critical set. The coefficients $\mathcal{Q}_{\J,k}(x,y)$  and the remainder term $\mathcal{R}_{\J,\tilde N}(x,y,\mu) $ depend smoothly on $R,s$, and $x,y \in Y \cap M_\mathrm{prin}$. Furthermore, they satisfy bounds analogous to the ones in (a), where now  derivatives in $t$ up to order $2k$ and $2\tilde N$ can occur, and 
the constants $C_{k,\Phi_{x,y}}$ and  $\widetilde C_{\tilde N,\Phi_{x,y}}$ are no longer uniformly bounded, but satisfy
\begin{align*}
C_{k,\Phi_{x,y}}& \ll  \dist(y, \O_x)^{-(n-1-\kappa)/2 -k}, \qquad 
\widetilde C_{\tilde N,\Phi_{x,y}} \ll \dist(y, \O_x)^{-(n-1-\kappa)/2-\tilde N}.
\end{align*}
%, and is  given by
%\bqn
%^0\Phi_{x,y}^\J(R,t)= R \, c_{x,g \cdot y}(t), \qquad c_{x,g \cdot y}(t):=\pm \frac{\norm {\kappa(x) - \kappa( g_\J \cdot y )}}{\norm {\grad_\eta \zeta(t,\kappa(x),\omega_\J) }}, \qquad (\omega_\J,g_\J) \in \J.
%\eqn
\end{enumerate}
\end{thm}

\begin{proof}
The asymptotic expansion for the integral $I^\gamma_{x,y}(\mu)$,  the smoothness of the coefficients $\mathcal{Q}_{k}(x,y)$, $\mathcal{Q}_{\J,k}(x,y)$, and the remainder terms  in the parameters $R,s$, and $x,y \in Y\cap M_\text{prin}$, as well as corresponding bounds for the coefficients and the remainder term were shown in \cite[Theorem 3.3]{ramacher16}. To improve on   the remainder estimate concerning its dependence on $\gamma$ as $\mu \to +\infty$, we rewrite  $I^\gamma_{x,y}(\mu)$  up to a volume factor as
\bqn
I^\gamma_{x,y}(\mu)\equiv \int_{\t}\int_{\Sigma^{R,s}_x} e^{i\mu \Phi_{x,y}(\omega,\exp (-X))} e^{-\gamma(X)} a (x,y, \omega,X) \d \Sigma^{R,s}_x (\omega) \d X, \qquad \gamma \in \widehat T,
\eqn
where we can assume that $a$ is compactly supported with respect to $X\in \t$ in a small open connected subset $^y\t\subset \t$   by choosing $Y$ small. 
If we were to apply the stationary and non-stationary phase principles to $I^\gamma_{x,y}(\mu)$ with  $\Phi_{x,y}$ as phase function, which was the way we followed in \cite{ramacher16}, this would involve derivatives of the amplitude $\overline \gamma a$ and generate non-optimal powers in $\gamma$ in the remainder estimate. Instead, note  that the character $\gamma(t)=e^{\gamma(X)}\in S^1$ constitutes  itself a phase, which can oscillate rather quickly as $\gamma$ increases. To deal with these oscillations, we shall  absorb them into the phase function, and define for arbitrary  $\xi \in \t^\ast$ 
\bqn 
 \Phi^\xi_{x,y}(\omega,X):=  \Phi_{x,y}(\omega,\e{-X})-\xi(X), \qquad t=\exp X \in T. 
\eqn
The idea is then to apply  the stationary and non-stationary phase principles to the integrals $I^\gamma_{x,y}(\mu)$ with phase function $\Phi^\xi_{x,y}(\omega,X)$ and $\xi=\gamma/i\mu$ as parameter, compare \cite[Theorem 7.7.6]{hoermanderI}, to obtain remainder estimates that are optimal in $\gamma\in \mathcal{V}_\mu$. 
If $\mklm{X_1,\dots,X_d}$ denotes a basis of $\t$, the $X$-derivatives of $ \Phi^\xi_{x,y}(\omega,X)$ read 
\bqn 
\sum_{k=1}^n \omega_k (d \tilde x_k)_{\e{-X} \cdot y} (\widetilde X_j) -\xi(X_j)=[\Jbb(\e{-X} \cdot y, \omega)-\xi](X_j),
\eqn
 so that 
 \bqn 
\Crit  \, \Phi^\xi_{x,y}=\mklm{(\omega,X) \mid \kappa(x) - \kappa (\e{-X} \cdot y)  \in N_\omega(\Sigma^{R,s}_x), \quad (\e{-X} \cdot y, \omega) \in \Jbb^{-1}(\mklm{\xi})}.
 \eqn
A repetition of the arguments given in \cite[Proof of Lemma 3.1]{ramacher16} then shows that for sufficiently small $|\xi|$ 
 \begin{itemize}
\item if  $y\in \mathcal{O}_x$, the set $\Crit \, \Phi^\xi_{x,y}$ is clean and given by the smooth submanifold
\bqn
%\label{eq:22.09.2015}
\J=\big \{(\omega,X)\mid (\e{-X} \cdot y, \omega) \in \Jbb^{-1}(\mklm{\xi}),  \,  x=\e{-X}\cdot y\big \}%= V_\J \times G_\J
\eqn
of  codimension $2\dim \O_x$; %with $V_\J=\Sigma^{R,s}_x \cap N_x\O_x$ and $G_\J=\mklm{g \in G \mid x=g\cdot y}\subset \, ^y G$.

\item if $y \not\in  \mathcal{O}_x$  and $T$ acts on $M$ with orbits of the same dimension $\kappa$ %that is, $M=M_\mathrm{prin}\,  \cup\,  M_\mathrm{except}$, 
and the co-spheres $S_x^\ast M$ are strictly convex, then either $\Crit \, \Phi^\xi_{x,y}$ is empty, or,  choosing $Y$ sufficiently small,  $\Crit \, \Phi^\xi_{x,y}$ is clean and of codimension $n-1+\kappa$,
% its finitely many connected components being of the form 
%\bqn 
%\J=V_\J \times G_\J
%\eqn
%with  $V_\J= \mklm{\omega_\J}$ and $G_\J = g_\J \cdot G_y\subset \, ^y G$ for some $\omega_\J\in \Sigma^{R,s}_x$ and $g_\J \in G$.  
\end{itemize}
which would also just follow from \cite[Proof of Lemma 3.1]{ramacher16} and the implicit function theorem.
In addition, note that  for $(\omega,X)\in \Crit \, \Phi^\xi_{x,y}$ 
\bqn
\M_{x,y}(\omega,X):=\text{Trans Hess } \Phi^\xi_{x,y}(\omega,X) \, \text{is independent of $\xi$}.
\eqn
%The idea is now to apply  the stationary and non-stationary phase theorems to the integrals $I^\gamma_{x,y}(\mu)$ with respect to the phase function $\Phi^{i\gamma/\mu}_{x,y}$ uniformly in $\gamma \in \mathcal{V}_\mu$. 
Next, notice that under the assumptions in (a) and (b), respectively, there is an open tubular neighborhood $U_0$  of $\Crit \, \Phi_{x,y}$ and  a constant $\mu_0>0$ such that for all $\mu \geq \mu_0$ and $\gamma \in \mathcal{V}_\mu$
\begin{itemize}
\item $\Crit \, \Phi_{x,y}^{\gamma / i\mu} \subset U_0$,
\item $\Crit \, \Phi_{x,y}^{\gamma/ i\mu}$ is clean, that is, $\Phi_{x,y}^{\gamma/i\mu} $ is a Morse-Bott function. 
\end{itemize}
Let  $U_1$ and $U_2$ be two further open tubular neighborhoods of  $\Crit \, \Phi_{x,y}$ and $\mu_0 >\mu_1 >\mu_2>0$  be such that $U \subset U_1 \subset U_2$ are proper inclusions and the pairs $(U_1,\mu_1)$, $(U_2,\mu_2)$ have the same properties than $(U_0,\mu_0)$. Let $u \in \Cinft(U_2,\R^+)$ be a test function with $u_{|U_1}\equiv 1$ and define
\begin{align*}
^1I^\gamma_{x,y}(\mu)&:= \int_{\t}\int_{\Sigma^{R,s}_x} e^{i\mu \Phi^{\gamma/i\mu}_{x,y}(\omega,X)}u(\omega,X) a(x,y, \omega,X) \d \Sigma^{R,s}_x (\omega) \d X, \\
^2I^\gamma_{x,y}(\mu)&:=I^\gamma_{x,y}(\mu)-^1I^\gamma_{x,y}(\mu).
\end{align*}
By construction,  for $\gamma \in\mathcal{V}_\mu$ and $\mu \geq \mu_0$ all critical sets $\Crit \, \Phi_{x,y}^{\gamma / i\mu}$  have a minimal, non-vanishing\footnote{ At least on the intersection of the support of $a(x,y,\cdot,\cdot)$ and $ U_1$.} distance to $\gd U_1$, so that 
\bqn 
|\grad \Phi^{\gamma/i\mu}_{x,y}| \geq C >0 \quad \text{on $\supp(1-u) a(x,y,\cdot,\cdot )$ for all $\gamma \in\mathcal{V}_\mu$ with $\mu \geq \mu_0$.}
\eqn
An application of the non-stationary phase principle  \cite[Theorem 7.7.1]{hoermanderI} with respect to the phase function $ \Phi^{\gamma/i\mu}_{x,y}$ then yields for every $k \in \N$ the uniform bound
\bqn 
^2I^\gamma_{x,y}(\mu) =O_{k,a}(\mu^{-k}) \qquad \text{for all $\gamma \in\mathcal{V}_\mu$ with $\mu \geq \mu_0$.}
\eqn
It remains to estimate the integral $^1I^\gamma_{x,y}(\mu)$ by means of the stationary phase principle with $\xi=\gamma/i\mu$ as parameter, for which we shall follow \cite[Theorem 7.7.5]{hoermanderI} and its proof.  Assume as we may that $U_2$ is sufficiently small, and introduce normal tubular coordinates on $U_2$ in form of an atlas $\mklm{(\zeta_\iota,\mathcal{Y}_\iota)}_{\iota \in I}$ such that
%\footnote{\bf Rethink carefully.  1 and 2 already follow by the implicit function theorem. Concering 3, the point should be that the critical sets are  rather explicitly determined...} 
 \begin{enumerate}
 \item $\supp ua(x,y,\cdot,\cdot ) \subset \bigcup_\iota \mathcal Y_\iota$,
 \item $\zeta_\iota^{-1}(m',m'') \in \Crit \, \Phi^\xi_{x,y}$ iff $\R^{\d''}\ni m''=m''_\xi$, where 
\bqn 
d''=\begin{cases} 2 \dim \O_x & \text{in case (a)} \\ n-1+\kappa & \text{in case (b)}. \end{cases}
\eqn
 \item the $\t$-coordinates  are given by standard Euclidean coordinates, so that in each chart 
 $$X=\sum_\alpha  m'_{\t,\alpha} X_\alpha ' + \sum_\beta m''_{\t,\beta} X_\beta ''$$
  for a suitable basis $\{X_\alpha',X_\beta''\}$ of $\t$. 
 \end{enumerate}
 Let $\mklm{p_\iota}$ be a partition of unity subordinated to the covering $\mklm{\mathcal{Y}_\iota}$, and write $  a_\iota(x,y,\omega,X):= p_\iota(\omega,X)  a(x,y,\omega,X) $  as well as  $a_\iota (x,y,m):=a_\iota (x,y, \zeta_\iota^{-1}(m))\beta_\iota(m) $, $\beta_\iota$ being a Jacobian.  Denote the product of  $u \circ \zeta^{-1}_\iota$ with the Taylor expansion of $a_\iota(x,y,\cdot )$ in the variable $m''$ at the point $m''_\xi$ of order $2k$ by $T^\xi_\iota(x,y,m)$, which is smooth and bounded in $\xi$.  Let $\M_{x,y}(\omega,X)$ be as above and set $\M^\iota_{x,y}(m',m''_\xi):=(\M_{x,y}  \circ \zeta_\iota^{-1})(m',m''_\xi)$. Since for sufficiently small $|m''-m''_\xi|$
 \bqn 
 \frac{|m''-m''_\xi|}{|\grad_{m''} \Phi^\xi_{x,y} (m',m''_
 \xi)|} \ll \norm{ \M^\iota_{x,y}(m',m''_\xi)^{-1}} \ll 1
 \eqn
for all $\xi$, \cite[Theorem 7.7.1]{hoermanderI} yields with respect to $\Phi^\xi_{x,y}(m):=(\Phi^\xi_{x,y}\circ \zeta_\iota^{-1})(m)$ for any $k \in \N$
 \bqn
 ^1I^\gamma_{x,y}(\mu)= \sum_\iota \int_{\R^{d'}} \int_{\R^{d''}} e^{i\mu \Phi^{i\gamma/\mu}_{x,y}(m)} T^{i\gamma/\mu}_\iota(x,y,m) \d m'' \d m' +O_{k,a}(\mu^{-k})
 \eqn
 uniformly in $\gamma$. Next, note that for fixed $m'$
 \bq
 \label{eq:quadrform}
m'' \longmapsto \eklm{\M^\iota_{x,y}(m',m''_\xi) (m''-m''_\xi),(m''-m''_\xi)}
 \eq
 defines a non-degenerate quadratic form, and introduce the auxiliary function 
 \bqn
 H^\xi(m):=\Phi^\xi_{x,y}(m)-\Phi^\xi_{x,y}(m',m''_\xi)-\eklm{\M^\iota_{x,y}(m',m''_\xi) (m''-m''_\xi),(m''-m''_\xi)}/2,
 \eqn
 which vanishes of third order at $m''=m''_\xi$. The function
  \bqn
 ^s\Phi^\xi_{x,y}(m):=\eklm{\M^\iota_{x,y}(m',m''_\xi) (m''-m''_\xi),(m''-m''_\xi)}/2+s H^\xi(m) 
 \eqn
 interpolates between $\Phi^\xi_{x,y}(m)-\Phi^\xi_{x,y}(m',m''_\xi)=\, ^1\Phi^\xi_{x,y}(m)$ and  the quadratic form \eqref{eq:quadrform}, and we define
 \bqn 
 \I(s):=\int_{\R^{d''}} e^{i\mu \, ^s\Phi^{\xi}_{x,y}(m)} T^\xi_\iota(x,y,m) \d m''.
 \eqn
 Taylor expansion then yields
 \bqn 
 \Big | \I(1)- \sum_{l=0}^{2k-1} \I^{(l)} (0)/l! \Big | \ll \sup_{0 \leq s \leq 1} |\I^{(2k)} (s)|/L!.
 \eqn
 Now, differentiation with respect to $s$ gives 
 \bqn
 \I^{(l)}(s)= \int_{\R^{d''}} e^{i\mu \, ^s\Phi^{\xi}_{x,y}(m)} (i\mu H^\xi(m))^l\, T^\xi_\iota(x,y,m) \d m''.
 \eqn
  In view of the uniform bounds
  \bqn 
 \frac{|m''-m''_\xi|}{|\grad_{m''} \, ^s\Phi^\xi_{x,y} (m',m''_\xi)|} \ll \norm{ \M^\iota_{x,y}(m',m''_\xi)^{-1}} \ll 1 \qquad \text{for all $\xi$ and $s$}
 \eqn
and\footnote{Note that $D^\alpha_{m''} H^\xi(m)=D^\alpha_{m''} \Phi_{x,y}(m)$ for $|\alpha|\geq 3$, while for $|\alpha| \leq 2$ Taylor expansion at $m''_\xi$ implies
\begin{align*}
|D^\alpha_{m''} H^\xi(m)|  \ll |m''-m''_\xi|^{3-|\alpha|} \sum_{|\beta|=3} \sup |D^\beta_{m''} \Phi_{x,y}(m)| \ll  |m''-m''_\xi|^{3-|\alpha|} 
\end{align*}
 uniformly in $\xi$ since $H^\xi(m)$ depends on $\xi$ only via the term $\xi\big (\sum_\alpha  m'_{\t,\alpha} X_\alpha ' + \sum_\beta m''_{\t,\beta} X_\beta ''\big )$, which vanishes when differentiated more than one time.} 
\bqn
\big |D^\alpha_{m''}  [H^\xi(m)^{2k}\, T^\xi_\iota(x,y,m)]\big | \ll |m''-m''_\xi|^{6k-|\alpha|} \quad \text{for all $\xi$} % with  $\gamma\in \mathcal{V}_\mu$ with $\mu \geq \mu_0$}
\eqn
we obtain from \cite[Theorem 7.7.1]{hoermanderI} with $k$ replaced by $3k$ there the important uniform bound
 \bqn
 \I^{(2k)}(s) =O(\mu^{-k}) \qquad \text{for all $\gamma \in \mathcal{V}_\mu$ with $\mu \geq \mu_0$ and  all $s$.}
 \eqn
 Next, denote by $\H^\xi(m)$ the Taylor expansion of $H^\xi(m)$ of order $3k$, and notice that one has
 \bqn 
 (H^\xi)^l-(\H^\xi)^l=O(|m''-m''_\xi|^{2k+2l})
 \eqn
 uniformly in $\xi$. Applying again \cite[Theorem 7.7.1]{hoermanderI} gives
 \bqn
 \I^{(l)}(0)= \int_{\R^{d''}} e^{i\mu \, ^0\Phi^{\xi}_{x,y}(m)} (i\mu \H^\xi(m))^l\, T^\xi_\iota(x,y,m) \d m'' + O_{k,a}(\mu^{-k})
 \eqn
 uniformly in $\xi$. The assertion now follows by taking into account \cite[Lemma 7.7.3]{hoermanderI} and the final arguments in the proof of  \cite[Theorem 7.7.5]{hoermanderI}.
Note that the Taylor expansion $\H^\xi$ starts with terms of degree $3$ and depends on $\xi$ in that the coefficients are evaluated at $m''=m''_{\xi}$. Consequently, when applied to $\I^{(l)}(0)$ the remainder estimate in \cite[Lemma 7.7.3]{hoermanderI} can be uniformly estimated in $\xi$.  The final remainder estimate results from the above uniform estimates, and  local contributions of higher order where additional derivatives of $\gamma$ arise. The local terms are unique, and coincide with the ones with phase function $\Phi_{x,y}$ and amplitude $\overline \gamma a$ considered in  \cite[Theorem 3.3]{ramacher16}, from which the corresponding bounds are deduced. The fact that in case (a) only $t$-derivatives of order $k$ and $\tilde N$ occur,  follows from the particular form of the transversal Hessian, \cite[Proof of Theorem 3.3]{ramacher16}. 
%\footnote{\bf Restructure proof...}
 \end{proof}

Similarly, one derives 

\begin{thm}
\label{thm:14.05.2017} 
Consider the integrals $I^\gamma_{x,y}(\mu)$  defined in \eqref{eq:03.05.2015}. Assume that the torus $T$ acts on $M$ with orbits of the same dimension $\kappa \leq n-1$,  and that the co-spheres $S_x^\ast M$ are strictly convex.  Then, for sufficiently small $Y$ and arbitrary $\tilde N_1, \tilde N_2 \in \N$   one has the asymptotic formula
\begin{gather*}
 I^\gamma_{x,y}(\mu)\\
 = \sum_{\J \in \pi_0(\Crit \, \Phi_{x,y})}   \frac{e^{i \mu  \,^0\Phi_{ x,y}^\J}}{\mu^\kappa (\mu \norm{ \kappa(x)-\kappa(g_\J \cdot y)}+1 )^{\frac{n-1-\kappa}2}} \left [ \sum_{k_1,k_2 =0}^{\tilde N_1-1,\tilde N_2-1}  \frac {   \mathcal{Q}_{\J, k_1,k_2} (x,y)}{\mu^{k_1} (\mu \norm{ \kappa(x)-\kappa(g_\J \cdot y)}+1)^{k_2}}\right. \\  \left. + \mathcal{R}_{\J,\tilde N_1, \tilde N_2}(x,y,\mu) \right ]
\end{gather*}
as $\mu \to +\infty$.  The coefficients  and the remainder term  depend smoothly on $R,t$, %with support in each of the components $\J$ of $\Crit  \, \Phi_{x,y}$ and $\Sigma^{R,s}_{x} \times T$, 
 while  $^0\Phi_{x,y}^\J:= R \,  c_{x,g_\J\cdot y} (t)$ denotes the constant value of $\Phi_{x,y}$ on $\J$. Furthermore, the coefficients are uniformly bounded in $R,s, x$, and $y$ by derivatives of $\gamma$   up to order $2k_1$,
 and the remainder term
 \bqn 
 \mathcal{R}_{\J,\tilde N_1, \tilde N_2}(x,y,\mu)=  O_{\J,\tilde N_1, \tilde N_2} \Big (   \mu^{-\tilde N_1}  (\mu \norm{ \kappa(x)-\kappa(g_\J \cdot y)}+1)^{-\tilde N_2} \Big ) 
 \eqn
 by derivatives of $\gamma$ up to order  $2 \tilde N_1$, provided that $\gamma \in \mathcal{V}_\mu$.
\end{thm}
\begin{proof} 
The proof is essentially the same than the one of \cite[Theorem 3.4]{ramacher16}, using the arguments given in the proof of the previous theorem.
%\footnote{\bf Check carefully!}
\end{proof}

\section{The equivariant local Weyl law}

We shall now prove  an improved version of the equivariant local Weyl derived in \cite{ramacher16}. For this, we first prove   the following refinement of \cite[Proposition 4.1]{ramacher16}.

\begin{proposition} [\bf Point-wise  asymptotics for the kernel of the equivariant approximate projection]
\label{thm:kernelasymp}
For any  fixed $x \in M$, $\gamma \in \widehat T$,  and $\tilde N\in \N$ one has as $\mu \to +\infty$
\begin{align}
\label{eq:13.05.2015}
\begin{split}
K_{\widetilde \chi_\mu \circ \Pi_\gamma}(x,x)&=\sum_{j\geq 0, \, e_j \in \L^2_\gamma(M)} \rho(\mu-\mu_j) |{e_j(x)}|^2 \\ & = \Big (\frac{\mu}{2\pi}\Big )^{n-\dim \mathcal{O}_x-1} \frac{d_\gamma}{2\pi} \left [\sum_{k=0}^{\tilde N-1} \Lcal_k(x,\gamma) \mu^{-k}+ \mathcal{R}_{\tilde N}(x,\gamma) \right ]
\end{split}
\end{align}
with coefficients  and  remainder  depending smoothly  on $x \in M_\mathrm{prin}$. They  satisfy the bounds
\bqn 
|\mathcal{L}_k(x,\gamma)| \leq C_{k,x} \sup_{l \leq k} \norm{D^l \gamma}_\infty, 
\eqn
as well as 
\bqn
|\mathcal{R}_{\tilde N}(x,\gamma)| \leq \tilde C_{\tilde N,x}  \sup_{l \leq \tilde N} \norm{D^l \gamma}_\infty \mu^{-\tilde N}, \qquad \gamma \in \mathcal{V}_\mu,
\eqn
where $D^l$ denotes a differential operator on $T$ of order $l$, and the constants $C_{k,x}$, $\tilde C_{\tilde N,x}$ are uniformly bounded in $x$ if $M= M_\mathrm{prin} \cup M_\mathrm{except}$. In particular, the leading coefficient is given by
\begin{align*}
\Lcal_0(x,\gamma) =  \hat \rho(0) [{\pi_\gamma}_{|T_x}:\1] \, \mbox{vol} \, [( \Omega \cap S_x^\ast M)/T],
\end{align*}
where $S^\ast M:=\mklm{(x,\xi) \in T^\ast M\mid p(x,\xi)=1}$. If $\mu \to -\infty$, the function $K_{\widetilde \chi_\mu \circ \Pi_\gamma}(x,x)$ is rapidly decreasing in $\mu$. 
\end{proposition}
\begin{proof}
We only have to prove the bounds for the coefficients and the remainder, since all other asssertions have been shown in \cite{ramacher16}. Let the notation be as in  Section \ref{sec:RSF}, and $R,s \in \R$, $x \in Y_\iota$ be fixed.  As a direct consequence of Theorem \ref{thm:12.05.2015} (a) we have for any $\tilde N\in \N$
\bqn 
 \gd_{R,s}^\beta I^\gamma_\iota(\mu, R, s, x,x)= (2\pi/\mu)^{\dim \mathcal{O}_x} \left [ \sum_{k=0}^{\tilde N-1} \Lcal^k_{\iota,\beta}(R,s,x,\gamma) \mu^{-k}+ \mathcal{R}^{\tilde N}_{\iota,\beta}(R,s,x,\gamma,\mu) \right ],
\eqn
where the coefficients  and the remainder term are explicitly given and depend smoothly on $R,s$, and $x\in Y \cap M_\mathrm{prin}$.  Furthermore, both the coefficients $\Lcal^k_{\iota,\beta}(R,s,x,\gamma)$  and the remainder are bounded by expressions involving derivatives of $\gamma$ up to order $k$ and $\tilde N$, respectively, which are uniformly bounded in $x$ if $M=M_\mathrm{prin} \cup M_\mathrm{except}$. Equation  \eqref{eq:13.06.2016} then  implies the asymptotic expansion \eqref{eq:13.05.2015}  with the specified estimate for the remainder.  

\end{proof}
%\begin{rem}
%\label{rem:12.7.2017} Note that, if $M= M_\mathrm{prin} \cup M_\mathrm{except}$,   the previous proposition and  the Cauchy-Schwarz inequality imply  for $\tilde N=0$ with $ \kappa :=\dim G/K$    the estimate
%\begin{align*}
%\begin{split}
%K_{\widetilde \chi_\mu \circ \Pi_\gamma}(x,y)&\leq  \sqrt{K_{\widetilde \chi_\mu \circ \Pi_\gamma}(x,x)} \sqrt{K_{\widetilde \chi_\mu \circ \Pi_\gamma}(y,y)}\ll\, \mu^{n-\kappa-1} \, d_\gamma \, \sup_{l \leq  \lfloor \kappa/2+1\rfloor} \norm{D^l \gamma}_\infty,
%\end{split}
%\end{align*}
%uniformly in $x$ and $y$, $\rho\in \S(\R,\R_+)$ being a positive function. 
%\end{rem}

We can now sharpen \cite[Theorem 4.3]{ramacher16} in the isotypic aspect as follows.

\begin{thm}[\bf Equivariant local Weyl law]  
\label{thm:main}
Let $M$ be  a closed connected Riemannian manifold $M$ of dimension $n$ carrying an isometric and effective action of a torus $T$, and $P_0$ a $T$-invariant elliptic classical pseudodifferential operator on $M$ 
of degree $m$. Let  $p(x,\xi)$ be its  principal symbol, and assume that $P_0$    is positive and symmetric. Denote its  unique self-adjoint extension by  $P$, and for a given $\gamma \in \widehat T$ let $e_\gamma(x,y,\lambda)$ be its reduced spectral  function. Further, let $\Jbb:T^\ast M \to \t^\ast$ be the momentum map of the $T$-action on $M$, and put $\Omega:=\Jbb^{-1}(\mklm{0})$.  Then, for fixed $x \in M$ one has
 \bq
 \label{eq:29.10.2015}
\left |e_\gamma(x,x,\lambda)-\frac{ [\pi_{\gamma|T_x}:\1]}{(2\pi)^{n-{\kappa_x}}}  \lambda^{\frac{n-\kappa_x}{m}} \int_{\mklm{\xi\mid \, (x,\xi) \in \Omega, \, p(x,\xi)< 1}} \frac{ \d \xi}{\vol \O_{(x,\xi)}} \right | \leq C_{x,\gamma} \, \lambda^{\frac{n-{\kappa_x}-1}{m}}
\eq
as $\lambda \to +\infty$,  where $\kappa_x:=\dim \O_x$ and  $ [\pi_{\gamma|T_x}:\1]\in \mklm{0,1}$ denotes the multiplicity of the trivial representation  in the restriction of $\pi_\gamma$ to the isotropy group $T_x$ of $x$. Furthermore,  for arbitrary $\gamma \in \W_\lambda:=\mklm{\gamma \in \widehat T' \mid |\gamma | \leq \frac{\lambda^{1/m}}{\log \lambda}}$
\bq
\label{eq:4.6.2017}
C_{x,\gamma}=O_x\Big (\sup_{l \leq 1 }\norm{D^l \gamma}_\infty\Big )=O_x  (|\gamma| )
\eq
 is a constant that depends smoothly on $x\in M_\mathrm{prin}$ and is uniformly bounded in $x$ if   $M= M_\mathrm{prin} \cup M_\mathrm{except}$.
\end{thm}
\begin{proof}
This follows directly  by taking $\tilde N=1$ in  \eqref{eq:13.05.2015} and integrating with respect to $\mu$ from $-\infty$ to $\sqrt[m]{\lambda}$ with  the arguments given in \cite[Proof of Eq. (2.25)]{duistermaat-guillemin75}.
\end{proof}

\begin{rem}
\label{rem:23.04.2017}
\hspace{0cm}
\begin{enumerate}
\item
With the same constant $C_{x,\gamma}$ as in  \eqref{eq:29.10.2015} one also has the bound 
\bqn 
\left | e_\gamma(x,y,\lambda+1)-e_\gamma(x,y,\lambda) \right | \leq \sqrt{C_{x,\gamma} \lambda^{\frac{n-\kappa_x-1}m}} \sqrt{C_{y,\gamma} \lambda^{\frac{n-\kappa_y-1}m}}, \qquad x, y \in M, \, \gamma \in \W_\lambda,
\eqn 
compare \cite[Remark 4.4]{ramacher16}.
\item As a  consequence of Theorem \ref{thm:main}, the constant $C_{x,\gamma}$ in \cite[Corollary 4.6]{ramacher16} can be improved accordingly, as well as all examples given in \cite[Section 4]{ramacher16}.
\end{enumerate}
\end{rem}

\section{Equivariant $\L^p$-bounds of eigenfunctions for non-singular group actions}
\label{sec:equivLp}

 Let the notation be as in the previous sections. As a consequence of the improved point-wise  asymptotics for the kernel of the equivariant approximate projection,   one obtains in the non-singular case the following sharpened  equivariant $\L^\infty$-bounds for  eigenfunctions.

\begin{proposition}[\bf $\L^\infty$-bounds for isotypic spectral clusters]
\label{thm:bounds}
Assume that $T$ acts on $M$ with orbits of the same dimension $\kappa$, and denote by $\chi_\lambda$ the spectral projection onto the sum of eigenspaces of $P$ with eigenvalues in the interval  $(\lambda, \lambda+1]$. Then, for any $\gamma \in \W_\lambda$, 
\bq
\label{eq:5}
\norm{(\chi_\lambda\circ \Pi_\gamma) u}_{\L^\infty(M)} \leq C (1+ \lambda)^{\frac{n-\kappa-1}{2m}} \norm{u}_{\L^2(M)}, \qquad u \in \L^2(M),
\eq
for a positive constant $C$ independent of $\gamma$. In particular, we obtain
\bqn 
\norm{u}_{\L^\infty(M)} \ll  \lambda^{\frac{n-\kappa-1}{2m}}
\eqn
for any eigenfunction $u \in \L^2_\gamma(M)$ of $P$  with eigenvalue $\lambda$ satisfying  $\norm{u}_{\L^2}=1$ and $\gamma \in \W_\lambda$.
\end{proposition}

\begin{proof}
By Proposition \ref{thm:kernelasymp} we have for $\gamma \in \W_\lambda$ the uniform bound 
\bqn
|K_{ \widetilde \chi_\lambda\circ \Pi_\gamma}(y,y)| \ll (1+\lambda)^{\frac{n-\kappa-1}m}, \qquad  y \in M=M_\mathrm{prin} \cup M_\mathrm{except}.
\eqn
The assertion now follows by a repetition of the arguments in the proof of \cite[Proposition 5.1 and Equation (5.4)]{ramacher16}.
\end{proof}

% \begin{rem}
%If $G$ is a connected compact semisimple Lie group, the bound \eqref{eq:1.6.2017} can be rewritten in terms of the highest weight $\Lambda_\gamma\in \t^\ast_\C$ of $\gamma \in \widehat G$, and we obtain
%\bqn 
%C_{\gamma} = O \Big (\sqrt{|\Lambda_\gamma|^{2|\Sigma^+|+\lfloor \kappa/2+1\rfloor} }\Big ),
%\eqn
%compare Remark \ref{rem:22.5.2017}.
%\end{rem}
%
%In what follows, we shall derive refined $\L^p$-bounds for isotypic spectral clusters using complex interpolation techniques. For this, we shall need  the additional assumption that the co-spheres $S_x^\ast M$ are strictly convex. In essence, the proof is an elaboration of arguments from \cite{seeger-sogge} applied to the equivariant setting. While for the proof of the $\L^\infty$-bounds in the previous proposition it was sufficient to consider the asymptotic behaviour of  the integrals $I^\gamma_{x,y}(\mu)$ in case that $x=y$, the proof of $\L^p$-estimates actually requires  estimates for the  integrals  $I^\gamma_{x,y}(\mu)$ in a neighborhood of the diagonal, making things significantly more involved. This leads us to our second main result.

Similarly, we are able to sharpen the  $\L^p$-bounds for isotypic spectral clusters derived in \cite[Theorem 5.4]{ramacher16} in the isotypic aspect. 

\begin{thm}[\bf $\L^p$-bounds for isotypic spectral clusters]
\label{thm:20.02.2016}
Let $M$ be  a closed connected Riemannian manifold $M$ of dimension $n$ on which a torus $T$ acts effectively and isometrically with orbits of the same dimension $\kappa$. Further, let $P$ be the unique self-adjoint extension of a $T$-invariant elliptic positive symmetric classical pseudodifferential operator on $M$ 
of degree $m$, and assume that its principal symbol $p(x,\xi)$ is such that the co-spheres $S_x^\ast M:=\mklm{(x,\xi) \in T^\ast M\mid \, p(x,\xi)=1}$ are strictly convex. Denote by $\chi_\lambda$ the spectral projection onto the sum of eigenspaces of $P$ with eigenvalues in the interval  $(\lambda, \lambda+1]$, and by $\Pi_\gamma$ the projection onto the isotypic component $\L^2_\gamma(M)$, where  $\gamma \in \widehat T$. Then,  for $u \in \L^2(M)$ and arbitrary  $\gamma\in \W_\lambda$
\bq
\label{eq:31.12.2015}
\norm{(\chi_\lambda \circ \Pi_\gamma) u}_{\L^q(M)} \leq \begin{cases} C \,  \lambda^{\frac{\delta_{n-\kappa}(q)}{m}} \norm{u}_{\L^2(M)}, &  \frac{2(n-\kappa+1)}{n-\kappa-1} \leq q \leq \infty, \vspace{2mm} \\ C \, \lambda^{\frac{(n-\kappa-1)(2-q')}{4m q'}} \norm{u}_{\L^2(M)}, &  2 \leq q \leq \frac{2(n-\kappa+1)}{n-\kappa-1}, \end{cases} 
\eq
 for a positive constant $C$ independent of $\gamma$,  where $\frac 1q+\frac 1{q'}=1$ and  
 \bqn 
 \delta_{n-\kappa}(q):=\max \left ( (n-\kappa) \left | \frac 12-\frac 1q \right| -\frac 12,0 \right ).
 \eqn
In particular, 
\bqn 
\norm{u}_{\L^q(M)} \ll \begin{cases}   \lambda^{\frac{\delta_{n-\kappa}(q)}{m}}, &  \frac{2(n-\kappa+1)}{n-\kappa-1} \leq q \leq \infty, \vspace{2mm} \\  \lambda^{\frac{(n-\kappa-1)(2-q')}{4m q'}}, &  2 \leq q \leq \frac{2(n-\kappa+1)}{n-\kappa-1}, \end{cases} 
\eqn
for any eigenfunction $u \in \L^2_\gamma(M)$ of $P$  with eigenvalue $\lambda$    satisfying $\norm{u}_{\L^2}=1$ and $\gamma \in \W_\lambda$.
\end{thm}

\begin{proof}
 The proof is a verbatim repetition of the proof of \cite[Theorem 5.4]{ramacher16} where instead of \cite[Theorem 3.4]{ramacher16} the improved estimates from  Theorem \ref{thm:14.05.2017} are used. 
 \end{proof}

As a consequence of the  previous theorem, all examples given in \cite[Section 5]{ramacher16} can be sharpened in the isotypic aspect.

\section{The singular equivariant local Weyl law. Caustics  and concentration of \\ eigenfunctions}
\label{sec:5}

Using the improved remainder estimates from  Theorem \ref{thm:12.05.2015} all results in \cite[Section 7]{ramacher16} can be sharpened. In particular,  the singular equivariant local Weyl law proved in \cite[Theorem 7.7]{ramacher16} can be improved in the isotypic aspect. 
% It describes the caustic behaviour of the reduced spectral function $e_\gamma(x,x,\lambda)$ near singular orbits. 
As before, let $M$ be a closed connected Riemannian manifold and $T$ a  torus  acting on $M$ by isometries, and consider the decomposition of $M$ into orbit types
 \bq
 \label{eq:2.19}
M=M(H_1) \, \dot \cup \, \cdots \, \dot \cup \, M(H_L),
\eq
where we suppose that  the isotropy types are numbered in such a way that $(H_i) \geq (H_j)$ implies $i \leq j$, $(H_L)$ being the principal isotropy type. We then have the following

\begin{thm}[\bf Singular equivariant local Weyl law]
\label{thm:15.11.2015}
Let $M$ be  a closed connected Riemannian manifold $M$ of dimension $n$ with an isometric and effective action of a torus $T$ and $P_0$ a $T$-invariant elliptic classical pseudodifferential operator on $M$ 
of degree $m$. Let  $p(x,\xi)$ be its  principal symbol, and assume that $P_0$    is positive and symmetric. Denote its  unique self-adjoint extension by  $P$, and for a given $\gamma \in \widehat T$ let $e_\gamma(x,y,\lambda)$ be its reduced spectral counting function. Write 
$\kappa$ for the dimension of an $T$-orbit in $M$  of principal type. Then, for $x \in M_\mathrm{prin}\cup M_\mathrm{except}$ one has the asymptotic formula
 \begin{gather*}
\left |e_\gamma(x,x,\lambda)- \frac{ \lambda^{\frac{n-\kappa}{m}}}{(2\pi)^{n-\kappa}} \sum_{N=1}^{\Lambda-1} \,  \sum_{i_1<\dots< i_{N}} \, \prod_{l=1}^{N}   |\tau_{i_l}|^{\dim G- \dim H_{i_l}-\kappa}  \mathcal{L}_{i_1\dots i_{N} }^{0,0}(x,\gamma)    \right | \\
\leq \widetilde C_\gamma \, \lambda^{\frac{n-\kappa-1}m} \sum_{N=1}^{\Lambda-1}\, \sum_{i_1<\dots< i_{N}}   \prod_{l=1}^N  |\tau_{i_l}|^{\dim G- \dim H_{i_l}-\kappa-1} 
\end{gather*}
 as $\lambda \to +\infty$, where  the multiple sum runs over all possible totally ordered subsets $\mklm{(H_{i_1}),\dots, (H_{i_N})}$ of singular isotropy types, and the coefficients satisfy the bounds
$
 \mathcal{L}_{i_1\dots i_{N}}^{0,0}(x,\gamma) \ll \norm{\gamma}_\infty
 $ 
 uniformly in $x$, while 
  \bqn 
 \widetilde  C_\gamma \ll  \sup_{l\leq 1} \norm{D^l \gamma}_\infty
  \eqn 
  is a constant independent of $x$ and $\lambda$,  the $D^l$ are differential operators on $T$ of order $l$, and the $\tau_{i_j}=\tau_{i_j}(x)$ parameters satisfying  $|\tau_{i_j}|\approx \dist (x, M(H_{i_j}))$.
\end{thm}
\begin{proof}
The proof consists in a verbatim repetition of the proof of \cite[Theorem 7.7]{ramacher16} using the improved remainder estimate in Theorem \ref{thm:12.05.2015} (a).
\end{proof}

As an immediate consequence this yields 

\begin{cor}[\bf Singular point-wise bounds for isotypic spectral clusters]
\label{cor:2.12.2015}
In the setting of Theorem \ref{thm:15.11.2015} we have 
\bqn 
\sum_{\stackrel{\lambda_j \in (\lambda,\lambda+1],}{ e_j \in \L^2_\gamma(M)}} |e_j(x)|^2  \leq \begin{cases}  C \, \lambda^{\frac{n-1}m}, & x\in M_\mathrm{sing}, \\
& \\
C_\gamma \,  \lambda^{\frac{n-\kappa-1}m} \sum\limits_{N=1}^{\Lambda-1}\, \sum\limits_{i_1<\dots< i_{N}}\prod\limits_{l=1}^N  |\tau_{i_l}|^{\dim G- \dim H_{i_l}-\kappa-1}, & x\in M-M_\mathrm{sing},  \end{cases}
\eqn
with $C>0$ independent of $\gamma$.  In particular, the bound holds for each individual $e_j \in \L^2_\gamma(M)$ with $\lambda_j \in (\lambda, \lambda+1]$. 
\end{cor}
\qed

%We would like to remark that the expansion in Theorem \ref{thm:15.11.2015} is only meaningful if $\lambda$ is sufficiently large compared to the  desingularization parameters $\tau_{i_l}$, more precisely, if
%\bqn 
%\lambda^{1/m} \prod_l |\tau_{i_l}|>1
%\eqn
%for all possible combinations of the $\tau_{i_l}$. While  \eqref{eq:29.10.2015} describes the asymptotics of the equivariant spectral function for arbitrary, but  fixed $x\in M$, Theorem \ref{thm:15.11.2015} gives a uniform description of  the  behaviour of the coefficients as $x\in M_\mathrm{prin}$ approaches singular orbits. 
%
%
% An asymptotic formula for $e_\gamma(x,x,\lambda)$ that  interpolates between the various asymptotic behaviours  in Theorem \ref{thm:main}, in the same way than Theorem \ref{thm:31.10.2015}   interpolates between the different asymptotics  in Theorem \ref{thm:12.05.2015} (a) can be obtained by integrating the expression for $K_{\Pi_\gamma  \circ \widetilde \chi_\mu} (x,x) $ in  Proposition \ref{prop:15.11.2015} with respect to $\mu$ from $-\infty$ to $\sqrt[m]\lambda$ for the values  $\eps=1$, $\tilde N_1=\kappa+1$, $\tilde N_2=1$ with the arguments given in the proof of Theorem \ref{thm:main}. This leads to expressions for  $e_\gamma(x,x,\lambda)$ which  involve the hypergeometric function, in the same way than the associated Legendre polynomials are given in terms of that function \cite[p.\ 188]{hobson}. 
% 

Integrating the asymptotic formulae in Theorems \ref{thm:main} and \ref{thm:15.11.2015} over $x\in M$ yields a sharpened remainder estimate for the equivariant Weyl law derived in \cite{ramacher10}. 
%Indeed, if  $N_\gamma(\lambda):= \int_M e_\gamma(x,x,\lambda) \d M(x)$ denotes the reduced counting function, one deduces
% \bqn
% \label{eq:weylglobal}
%N_\gamma(\lambda)= \frac{d_\gamma [{\pi_\chi}_{|H_L}:\1]}{(n-\kappa)(2\pi)^{n-\kappa}}   \mathrm{vol} \, [(\Omega \cap S^\ast M)/T]  \,  {\lambda} ^{\frac{n-\kappa}m }   + O\Big (d_\gamma\sup_{l \leq 1 }\norm{D^l \gamma}_\infty \lambda^{{(n-\kappa-1)}/m} (\log \lambda)^{\Lambda} \Big ).
%\eqn
In addition, as a consequence of the  previous theorem, the example given in \cite[Section 7]{ramacher16} can be sharpened in the isotypic aspect.

\section{Sharpness}
\label{sec:sharpness}

By the arguments given in  \cite[Section 8]{ramacher16} the remainder estimates in  Theorems \ref{thm:main} and  \ref{thm:15.11.2015}   are sharp in the spectral parameter $\lambda$, and already attained on the $2$-dimensional sphere $S^2$. %, but not optimal in the desingularization parameters $\tau_{i_j}$. 
To see that they are almost sharp in the isotypic aspect,   endow $M=S^2$ with the induced metric, and let $\Delta$ be the corresponding Laplace-Beltrami operator. The eigenvalues of $-\Delta$ are given by the numbers $ \lambda_k=k(k+1)$ with $k=0,1,2,3,\dots$,  and the corresponding $k(k+1)$-dimensional eigenspaces $\H_k$ are spanned by the classical spherical functions $Y_{km}$, $m \in \Z$, $|m| \leq k$.
%, so that 
%\bqn 
%-\Delta Y_{kl} = \lambda_k \, Y_{kl}.
%\eqn
The ${Y_{kl}}$ are orthonormal to each other, and by the spectral theorem we have the decomposition $
\L^2(M)= \bigoplus _{k=0}^\infty \H_k$. Furthermore, by restricting the left regular representation of $\SO(3)$ in $\L^2(S^2)$ to the eigenspaces $\H_k$ one obtains realizations for all elements in the unitary dual $\widehat{\SO(3)}\simeq \mklm{k=0,1,2,3,\dots}$. Now, let  $T= \SO(2)$ be isomorphic to  the isotropy group of a point in $S^2\simeq \SO(3)/\SO(2)$. The irreducible representations of $\SO(2)$ are $1$-dimensional, and the corresponding characters are given by the exponentials $\theta \mapsto e^{im\theta}$, where  $\theta \in [0,2\pi)\simeq \SO(2)$, $m \in \Z\simeq \widehat{\SO(2)}$. Each $\H_k$  decomposes into $\SO(2)$ representations with multiplicity $1$ according to  
$ 
\H_k=\bigoplus_{|m|\leq k} \H_k^m,
$ where  $\H_k^m$ is spanned by $Y_{km}$. 
Consequently, if $N_{m}(\lambda):=\int_{S^2} e_m(x,x,\lambda) dS^2(x)$ denotes the equivariant counting function   of $\Delta$  we obtain the estimate
\begin{align}
\label{eq:3.6.2017}
N_{m}(\lambda) =\sum_{k(k+1) \leq \lambda, \, |m| \leq k} 1\approx \sum_{|m| \leq k \leq \sqrt{\lambda}} 1\approx \sqrt{\lambda}-|m|,
\end{align}
as $\lambda \to +\infty $, 
%From this one recovers the classical Weyl law
%\[
%N(\lambda)= \sum_{k(k+1) \leq \lambda} \dim \mathcal{H}_k =\sum_{\gamma \in \widehat G} N_{\gamma}(\lambda)\approx \sum_{|m| \leq \sqrt{\lambda}} (\sqrt{\lambda} -|m|)\approx (2\sqrt\lambda +1)\sqrt \lambda -2 \frac{\sqrt \lambda(\sqrt \lambda+1)}2=\lambda.
%\]
%The asymptotic formula  \eqref{eq:3.6.2017} implies  that  the equivariant Weyl law proved in \cite[Theorem 9.5] {ramacher10} is sharp up to a logarithmic factor in the remainder estimate, 
showing that  the remainder estimates in  Theorems \ref{thm:main} and  \ref{thm:15.11.2015}   are  almost sharp both in the eigenvalue and in the isotypic aspect. \\

To see that the  equivariant $\L^p$-bounds in Section \ref{sec:equivLp} are almost sharp  in the eigenvalue  and  isotypic aspect, let us consider the standard $2$-torus $M=T^2\subset \R^3$ on which $G=\SO(2)$ acts by rotations around the symmetry axis. Then all orbits are $1$-dimensional and of principal type.
%, and Theorem \ref{thm:main} yields with the identification $\Z\simeq \widehat{\SO(2)}$ for the reduced spectral function of the Laplace-Beltrami operator
% \bqn
%e_m(x,x,\lambda)-\frac{1}{2\pi} \sqrt \lambda  \int_{\mklm{\xi\mid \, (x,\xi) \in \Omega, \, p(x,\xi)< 1}} \frac{ \d \xi}{\vol \O_{(x,\xi)}} =O(1+|m|^{1}),  \qquad m \in \Z,
%\eqn
%uniformly in $x \in T^2$,
 Proposition \ref{thm:bounds} then implies the bound
\bqn 
\norm{u}_{\L^\infty(T^2)} =O(1 ) , \qquad u \in \L^2(T^2), \, \norm{u}_{\L^2}=1,
\eqn
for any eigenfunction  of the Laplace-Beltrami operator $\Delta$ on $T^2$. Now, via  the identification 
\bqn 
\R^2/\Z^2 \stackrel{\simeq} \longrightarrow T^2 \simeq S^1 \times S^1, (x_1,x_2) \, \longmapsto \, (e^{2\pi i x_1}, e^{2\pi i x_2}),
\eqn
the  standard orthonormal basis of eigenfunctions of $\Delta$ is given by $\mklm{e^{2\pi i k_1 x_1}e^{2\pi i k_2 x_2}\mid (k_1,k_2) \in \Z^2}$, showing that the  bounds in Proposition \ref{thm:bounds} and Theorem \ref{thm:20.02.2016}  are almost sharp both in the eigenvalue  and  isotypic aspect.

\appendix
\renewcommand*{\thesection}{\Alph{section}}

%\section{Errata}
%
%Concerning connectedness... 

\providecommand{\bysame}{\leavevmode\hbox to3em{\hrulefill}\thinspace}
\providecommand{\MR}{\relax\ifhmode\unskip\space\fi MR }
\providecommand{\MRhref}[2]{%
  \href{http://www.ams.org/mathscinet-getitem?mr=#1}{#2}
}
\providecommand{\href}[2]{#2}

%\bibliography{bibliography}
%\bibliographystyle{amsplain}

\end{document}